\documentclass[english]{amsart}

\usepackage{amsmath}
\usepackage{amssymb}
\usepackage{amsthm}
\usepackage{amscd}
\usepackage{babel}
\usepackage[latin1]{inputenc}
\usepackage{graphicx}

\newcommand{\Z}[1]{\mathbb{Z}/#1\mathbb{Z}}

\newtheorem{Theorem}{Theorem}

\newtheorem{cor}{Corollary}
\newtheorem{prop}{Proposition}

\newtheorem{Lemma}{Lemma}

\theoremstyle{definition}
\newtheorem{defi}{Definition}
\newtheorem*{rem}{Remark}

\renewcommand{\subjclassname}{AMS \textup{2010} Mathematics Subject
Classification\ }

\author{Jos\'{e} Mar\'{i}a Grau}
\address{Departamento de Matemáticas, Universidad de Oviedo\\ Avda. Calvo Sotelo s/n, 33007 Oviedo, Spain}
\email{grau@uniovi.es}

\author{Celino Miguel}
\address{Instituto de Telecomunicaoes, Polo de Covilha}
\email{celino@ubi.pt}

\author{Antonio M. Oller-Marc\'{e}n}
\address{Centro Universitario de la Defensa de Zaragoza\\ Ctra. Huesca s/n, 50090 Zaragoza, Spain}
\email{oller@unizar.es}

\title{On the structure of quaternion rings over $\mathbb{Z}/n\mathbb{Z}$}

\begin{document}

\begin{abstract} In this paper we define $\left(\frac{a,b}{\Z{n}}\right)$, the quaternion rings over $\Z{n}$, and investigate their structure. It is proved that these rings are isomorphic to $\left(\frac{-1,-1}{\Z{n}}\right)$ if $ a \equiv b\equiv -1 \pmod{4}$ or to $\left(\frac{1,1}{\Z{n}}\right)$ otherwise. We also prove that the ring $\left(\frac{a,b}{\Z{n}}\right)$ is isomorphic to $\mathbb{M}_2(\Z{n})$ if and only if $n$ is odd and that all quaternion algebras defined over $\Z{n}$ are isomorphic if and only if $n \not \equiv 0 \pmod{4}$.

\end{abstract}

\maketitle
\subjclassname{11R52,16-99}


\section{Introduction}

The origin of quaternions dates back to 1843, when Hamilton considered a $4-$dimensional vector space over $\mathbb{R}$ with basis $\{1,i,j,k\}$ and defined an associative product given by the now classical rules $i^2=j^2=-1$ and $ij=-ji=k$. These ``Hamilton quaternions'' turned out to be the only division algebra over $\mathbb{R}$ with dimension greater than 2.

This idea was later extended to define quaternion algebras over arbitrary fields. Thus, a quaternion algebra over an arbitrary field $F$ is just a $4-$dimensional central simple algebra over $F$. This definition leads to different presentations according to the characteristic of the field $F$. If $F$ is a field of characteristic not $2$ a quaternion algebra over $F$ is a $4-$dimensional algebra over $F$ with a basis $\{1,i,j,k\}$ such that $i^2=a$, $j^2 = b$ and $ij = -ji=k$ for some $a,b\in F\setminus\{0\}$. On the other hand, if $F$ is a field of characteristic 2, a quaternion algebra over $F$ is a $4-$dimensional algebra over $F$ with a basis $\{1,i,j,k\}$ such that $i^2 + i = a$, $j^2 = b$, and $ ji = (i + 1)j=k$ for some $a\in F$ and $b\in F\setminus\{0\}$. The structure of quaternion algebras over fields is well-known. Indeed, such an algebra is either a division ring or isomorphic to the matrix ring $\mathbb{M}_2(F)$.

Generalizations of the notion of quaternion algebra to other commutative base rings $R$ have been considered by Hahn \cite{ha}, Kanzaki \cite{ka}, Knus \cite{max}, Gross and Lucianovic \cite{GLU} and most recently by John Voight \cite{jv, jv3}.

 On the other hand, quaternions over finite rings have attracted significant attention since they have applications in coding theory \cite{oz2,oz1,codes}.

En este trabajo consideramos una generalización de los quaternion algebras over a commutative ring with identity $R$ en una dirección diferente, en la linea de la original (o genuina lo que te parezca) de Hamilton sin más restricciones que $i^2$ and $j^2$ sean unidades del anillo $R$;
  Esto se distancia de las definiciones modernas de quaternion algebras pues en el caso de característica 2 el algebra resultante es conmutativa y más generalmente en característica potencia de dos, aunque no es conmutativa, no es central; i.e. su centro contiene estrictamente a $R$.

In particular, we will define quaternion rings over commutative, associative, unital rings as follows.

\begin{defi}
\label{def}
Let $R$ be a commutative and associative ring with identity and let $H(R)$ denote the free $R$-module of rank $4$ with basis $\{1, i, j, k\}$. That is,
\begin{equation*}H(R)=\{x_0+x_1i+x_2j+x_3k\;:\;x_0, x_1, x_2, x_3\in R\}.\end{equation*}
Now, let $a,b\in R$ be units and define an associative multiplication in $H(R)$ according to the following rules:
\begin{align*}i^2&=a,\\ j^2&=b,\\ ij&=-ji=k\end{align*}
Thus, we obtain an associative, unital ring which is denoted by $\left(\frac{a,b}{R}\right)$ to which we will refer as a quaternion ring over $R$.
\end{defi}

\begin{rem}
If $a=b=-1$, the corresponding quaternion ring in called the ring of Hamilton quaternions over $R$ and it is denoted by $\mathbb{H}(R)$.
\end{rem}

The following concepts extend the classical ones to this general setting.

\begin{defi}\label{trn}
Let $z=x_0+x_1i+x_2j+x_3k\in\left(\frac{a,b}{R}\right)$.
\begin{itemize}
\item[i)] The conjugate of $z$ is: $\bar{z}=x_0-x_1i-x_2j-x_3k$.
\item[ii)] The norm of $z$ is $\textrm{n}(z)=z\bar{z}=x_0^2-ax_1^2-bx_2^2-abx_3^2\in R$.
\item[iii)] The trace of $z$ is $\textrm{tr}(z)=z+\bar{z}=2x_0\in R$.
\end{itemize}
\end{defi}
%
%
 It is easy to see that the known characterization for quaternion rings over fields is no longer true in this general setting, even in characteristic  different from two. For instance, consider the ring of Hamilton quaternions over $\mathbb Z$. Clearly, the corresponding quaternion ring $\mathbb{H}(\mathbb Z)$ is not a division ring. On the other hand, we see that this ring is not isomorphic to the matrix ring $\mathbb{M}_2(\mathbb Z)$. To do so, let us consider $z\in\mathbb{H}(\mathbb Z)$ such that $z^2=0$. Then $n(z)=0$ and it follows that $z=0$. Since this property ($z^2=0\rightarrow z=0$) does not hold in $\mathbb{M}_2(\mathbb Z)$ both rings are not isomorphic, as claimed.

The question naturally arises as to whether a quaternion ring over an associative and commutative ring with identity $R$ is isomorphic to the matrix ring $\mathbb{M}_2(R)$. In this paper we consider the case $R=\mathbb{Z}/n\mathbb{Z}$. In particular we prove that, given $n \in \mathbb{N}$, there exist at most two quaternion rings over $\mathbb{Z}/n\mathbb{Z}$ up to isomorphism: $\mathbb{H}(\Z{n})$ and $\left (\frac{1,1}{\Z{n}}\right)$. Moreover, we will see that $\mathbb{H}(\Z{n})\cong \left (\frac{1,1}{\Z{n}}\right)\cong \mathbb{M}_2(\Z{n})$ if and only if $n$ is odd.

Note that if $n=p_1^{r_1}\ldots p_k^{r_k}$ is the prime factorization of $n$, then by the Chinese Remainder Theorem we have that
\begin{equation}\label{l3}\mathbb Z/ n \mathbb{Z}\cong\mathbb Z/{p_1^{r_1}\mathbb{Z}}\times\ldots\oplus\mathbb Z/{p_k^{r_k}\mathbb{Z}}.\end{equation}
Decomposition (\ref{l3}) induces a natural isomorphism
\begin{equation} \label{F1}\left(\frac{a,b}{\Z{n}} \right) \cong \left(\frac{a,b}{\Z{p_1^{r_1}}} \right) \oplus\ldots\oplus \left(\frac{a,b}{\Z{p_k^{r_k}}} \right) .\end{equation}
Consequently, it suffices to study the case when $n$ is a prime-power.

This fact strongly determines the structure of the paper. In section \ref{SEC:POWT} we focus on the case when $n$ is a power of two, while Section \ref{SEC:ODD} is devoted to the odd prime-power case. Before them, Section \ref{SEC:NUM} presents some auxiliary results from Elementary Number Theory that are useful in the sequel and in Section \ref{SEC:HAM} we study Hamilton quaternions over $\Z{n}$ and $\left(\frac{ 1, 1}{\Z{n}}\right)$ due to the main role that these particular cases will play in our classification.

\section{Some number-theoretical auxiliary results}
\label{SEC:NUM}

In this section we collect some results that will be useful in forthcoming sections. They are mainly related to finding solutions to quadratic polynomial congruences in two variables modulo a prime-power.

When we deal with polynomial congruences in one variable, Hensel's lemma plays a key role. The simplest form of Hensel's lemma \cite[p.170]{ROS} states that, under certain regularity conditions, a solution of a polynomial with integer coefficients modulo a prime number $p$ can be lifted to a solution modulo $p^j$ for $j>1$. The following lemma generalizes this result to polynomials in two variables.

\begin{Lemma}
\label{l1}
Let $f(x_1, x_2)$ be a  polynomial in two variables  with integer coefficients. let $p$ be a prime number, and let $A=(a_1, a_2)
\in\mathbb Z^2$ be such that
\begin{equation*}f(a_1, a_2)=0\;\;\;\;\;mod\;p^j,\end{equation*}
with at least one of the partial derivatives $\frac{\partial f}{\partial x_i}$ is nonzero at $(a_1, a_2)$ modulo $p$.
Then, there exist integers $t_1, t_2$ such that
\begin{equation*}f(a_1+t_1p^j, a_2+t_2p^j)=0\;\;\;\;\;mod\;p^{j+1}.\end{equation*}
\end{Lemma}
\begin{proof}
Let $n$ be the degree of the polynomial $f$. Then, using  Taylor's theorem for functions of two independent variables    we get
\begin{align*}f(a_1+t_1p^j, a_2+t_2p^j)=f(a_1, a_2)+t_1p^j\frac{\partial f}{\partial x_1}(a_1,a_2)+t_2p^j\frac{\partial f}{\partial x_1}(a_1,a_2)+ \\
+\frac{1}{2!}\left(t_1^2p^{2j}\frac{\partial^2 f}{\partial x_1^2}(a_1,a_2)+
2t_1t_2p^{2j}\frac{\partial^2 f}{\partial x_2 x_1}(a_1,a_2)+
t_2^2p^{2j}\frac{\partial^2 f}{\partial x_2^2}(a_1,a_2)\right) +\ldots \\
+ \frac{1}{n!}\left(t_1^np^{nj}\frac{\partial^n f}{\partial x_1^n}(a_1,a_2)+\binom {n} {1}t_1^{n-1}p^{(n-1)j}t_2p^{j}\frac{\partial^{n} f}{\partial x_2 x_1^{n-1}}(a_1,a_2)+\ldots
t_2^{n}p^{nj}\frac{\partial^n f}{\partial x_2^n}(a_1,a_2)\right).
\end{align*}
It is easy to check that each derivative $\frac{\partial^s f}{\partial x_2^{s-r}x_1^r}$ is divisible by $r!(s-r)!$. That is, $\frac{\partial^s f}{\partial x_2^{s-r}x_1^r}=r!(s-r)!g(x,y)$ for some polynomial $g$. Therefore,
\begin{equation*}\binom {s} {r}\frac{\partial^s f}{\partial x_2^{s-r}x_1^r}=\binom {s} {r}r!(s-r)!g(x,y)=s!g(x,y).\end{equation*}
It follows that modulo $p^{j+1}$ the Taylor expansion reduces to
\begin{equation}\label{l2}f(a_1+t_1p^j, a_2+t_2p^j)=f(a_1, a_2)+t_1p^j\frac{\partial f}{\partial x_1}(a_1,a_2)+t_2p^j\frac{\partial f}{\partial x_1}(a_1,a_2).\end{equation}
Finally, we observe that the assumption that at least one of the partial derivatives $\frac{\partial f}{\partial x_i}$ is nonzero at $(a_1, a_2)$ modulo $p$
 imply that we can choose $t_1$ and $t_2$ satisfying $t_1p^j\frac{\partial f}{\partial x_1}(a_1,a_2)+t_2p^j\frac{\partial f}{\partial x_1}(a_1,a_2)=-
  f(a_1, a_2)$. This completes the proof.
\end{proof}

Hence, to find solutions modulo $p^j$ it is enough to find solutions modulo $p$. The following Lemma \cite[p. 157]{OME} deals with the existence of solutions when $p$ is odd.

\begin{Lemma}
\label{l2}
Let $p$ be an odd prime and let $a,b$ be integers such that $\gcd(p,a)=\gcd(p,b)=1$. Then, the equation
$$ax^2+by^2\equiv \alpha \pmod{p}$$
has solutions for every $\alpha\in\mathbb{Z}$.
\end{Lemma}

We can combine the previous results in the following proposition.

\begin{prop}
\label{solp}
Let $p$ be an odd prime number and let $a,b,c\in\mathbb{Z}$ be coprime to $p$. Then, the congruence
$$ax^2+by^2\equiv c\pmod{p^s}$$
has a solution for every $s\geq 1$.
\end{prop}
\begin{proof}
Lemma \ref{l2} determines that there exists $(a_1,a_2)$ a solution to the congruences for $s=1$. Moreover, since $p\nmid c$ either $a_1$ or $a_2$ is coprime to $p$. Hence, Lemma \ref{l1} applies.
\end{proof}

Unfortunately, when $p=2$ we can never apply Lemma \ref{l1}. Consequently we can no longer provide a unified approach. The following results deal with some congruences that we will need to solve (in fact we will just need to know that they have a solution) in the sequel.

\begin{prop}
\label{lem4}
Let $a,b\in\mathbb{Z}$ be odd integers with $a \equiv b \pmod{8}$. Then the congruence
$$ax^2 \equiv b \pmod {2^s}$$
has a solution for every $s\geq 1$.
\end{prop}
\begin{proof}
If $1\leq s\leq 2$ the result follows by direct inspection. Now, let us assume that $s\geq 3$. Since $a$ is odd, let $\alpha$ be the inverse of $a$ modulo $2^s$. We have that $a\alpha\equiv 1\pmod{8}$ and hence $b\alpha\equiv 1\pmod{8}$. This means that $b\alpha=8k+1$ and congruence $ax^2\equiv b\pmod{2^s}$ becomes $x^2\equiv 8k+1\pmod{2^s}$. The result follows because $8k+1$ is a quadratic residue modulo $2^s$ if $s\geq 3$.
\end{proof}

\begin{prop}
\label{lem10}
The congruence
$$5x^2+5y^2\equiv 1\pmod{2^s}$$
has a solution for every $s\geq 1$.
\end{prop}
\begin{proof}
Given $s\geq 1$ let us denote by $\alpha_s$ the inverse of $5$ modulo $2^s$. The original congruence is equivalent to $x^2+y^2\equiv\alpha_s\pmod{2^s}$.

It can easily be seen that $\alpha_1=\alpha_2=1$, $\alpha_3=5$ and that, for every $k\geq 1$:
$$\alpha_{4k}=\alpha_{4k+1}=\alpha_{4k+3}=\frac{2^{4k+2}+1}{5},$$
$$\alpha_{4k+3}=\alpha_{4k+2}+2^{4k+2}.$$

Now, we claim that every prime divisor of $2^{4k+2}+1$ is of the form $4h+1$: let $p$ be a prime divisor of $2^{4k+2}+1$. Then $2^{4k+2}\equiv -1\pmod{p}$ and the order of $2$ in $\Z{p}$ must be a divisor of $8k+4$ not dividing $4k+2$. This means that the order of $2$ in $\Z{p}$ must be of the form $4h$ with $h\mid 2k+1$. Hence, $4h\mid p-1$ and $p=4h+1$ as claimed.

This implies that $\alpha_{4k},\alpha_{4k+1}$ and $\alpha_{4k+2}$ are the sum of two squares so \emph{a fortiori} the congruence $x^2+y^2\equiv\alpha_s\pmod{2^s}$ has a solution if $s=4k,4k+1,4k+2$.

We know that there exist $A,B\in\mathbb{Z}$ such that $A^2+B^2=\alpha_{4k+2}$ and we can assume, without loss of generality, that $A$ is odd. Let $a$ be the inverse of $A$ modulo $2^{4k+3}$. Then:
\begin{align*}
(A+2^{4k+1}a)^2+B^2&=A^2+B^2+2^{8k+2}a^2+2^{4k+2}\equiv\\ &\equiv A^2+B^2+2^{4k+2}=\alpha_{4k+2}+2^{4k+2}=\\ &=\alpha_{4k+3}\pmod{2^{4k+3}}.
\end{align*}

Consequently, the congruence $x^2+y^2\equiv\alpha_s\pmod{2^s}$ has a solution if $s=4k+2$ and the result follows.
\end{proof}

\section{Hamilton quaternions  over $ \mathbb{Z}_n$ and $\left(\frac{ 1, 1}{\Z{n}}\right)$}
\label{SEC:HAM}

It is well-known that Hamilton quaternions over the real numbers form an $\mathbb{R}-$algebra isomorphic to a subalgebra of the matrix algebra $\mathbb{M}_2(\mathbb{C})$, where the isomorphism is given by:
$$\mathbb{H}(\mathbb{R}):=\left(\frac{-1,-1}{\mathbb{R}}\right) \cong \left\{ \begin{pmatrix} z & w \\ -\overline{w} & \overline{z} \\ \end{pmatrix} : z,w \in \mathbb{C} \right\}.$$

In the same way, it is easy to observe that $\left(\frac{ 1, 1}{\mathbb{R}}\right)$ is also isomorphic to a subalgebra of complex matrices. Namely:
$$\mathbb{L}(\mathbb{R}):=\left(\frac{1,1}{\mathbb{R}}\right) \cong \left\{ \begin{pmatrix} z & w \\ \overline{w} & \overline{z} \\ \end{pmatrix} : z,w \in \mathbb{C} \right\}.$$

These isomorphisms are also valid if we consider the quaternion rings over an arbitrary commutative, associative, unital ring. We just have to replace $\mathbb{C}$ by the quotient ring $R[i]/\langle i^2+1\rangle$. In particular:
$$ \mathbb{H}(R):=\left(\frac{-1,-1}{R}\right) \cong \left\{ \begin{pmatrix} \alpha-\beta i & -\gamma +\delta i \\ \gamma +\delta i & \alpha+\beta i \\ \end{pmatrix} : \alpha,\beta,\gamma,\delta \in R\right\},$$
$$\mathbb{L}(R):=\left(\frac{1,1}{R}\right) \cong \left\{  \begin{pmatrix} \alpha-\beta i &  \gamma +\delta i \\ \gamma -\delta i & \alpha+\beta i \\ \end{pmatrix} : \alpha,\beta,\gamma,\delta \in R\right\}.$$

These isomorphisms turn out to be a very useful tool from the computational point of view when we deal with quaternions over $\Z{n}$. Now, we will have a close look at the rings $\mathbb{H}(\Z{n})$ and $\mathbb{L}(\Z{n})$. Recall that the natural isomorphism (\ref{F1}) allows us to focus on the prime-power case.

\subsection{The odd prime power case}

Hamilton quaternions over the field $\Z{p}$ have been studied in \cite{CEL}. Indeed, in
\cite{CEL} it is constructed an isomorphism between the Hamilton quaternions $\mathbb H(\Z{p})$ and the matrix ring $\mathbb{M}_2(\Z{p})$, for a given odd prime $p$. Here we generalize this result to Hamilton quaternions over $\Z{p^s}$ with $p$ an odd prime and $s\geq 1$.

\begin{prop}
Let $p$ be a odd prime number. Then,
$$\mathbb{H}(\Z{p^s})\cong \mathbb{L}(\Z{p^s})\cong \mathbb{M}_2(\Z{p^s})$$
for every $s\geq 1$.
\end{prop}
\begin{proof}
Due to Proposition \ref{solp} the congruence $x^2+y^2\equiv -1\pmod{p^s}$ has a solution for every $s\geq 1$. Let $a, b\in\Z{p^s}$ such that $a^2+b^2=-1$ and define an algebra homomorphism $\phi: \mathbb{H}(\Z{p^s})\longrightarrow \mathbb{M}_2(\Z{p^s})$ by:
$$\phi(i)=\begin{pmatrix}0&1\\ -1&0\end{pmatrix},\quad \phi(j)=\begin{pmatrix}a&b\\b& -a\end{pmatrix}.$$

The system of linear equations associated to
$$\phi(x_0+x_1i+x_2j+x_3k)=\begin{pmatrix}X&Y\\Z&T\end{pmatrix}.$$
has always a solution, namely:
\begin{align*}
x_0&=(X+T)/2,\\ x_1&=(Y - Z)/2,\\ x_2&=(a T - a X - b  Y - b Z)/2,\\ x_3&=(b  T - b  X + a Y + a Z)/2.\end{align*}
Hence, $\phi$ is an isomorphism.

The case $\mathbb{L}(\Z{p^s})$ is completely analogous considering $a, b\in\Z{p^s}$ such that $a^2+b^2=1$. In this case the isomorphism is given by:
$$\phi(i)=\begin{pmatrix}0&1\\ 1&0\end{pmatrix},\qquad \phi(j)=\begin{pmatrix}a&b\\b& -a\end{pmatrix}.$$
\end{proof}

\begin{cor}
Let $n$ be a odd integer. Then,
$$\mathbb{H}(\Z{n})\cong \mathbb{L}(\Z{n})\cong \mathbb{M}_2(\Z{n}).$$
\end{cor}

\subsection{The power of two case}
It is clear that $\mathbb{H}(\Z{2})\cong \mathbb{L}(\Z{2})$. Now we will se that if $s>1$, then $\mathbb{H}(\Z{2^s})\not\cong \mathbb{L}(\Z{2^s})$. To do so we first focus on the case $s=2$.

\begin{Lemma} \label{cuat}
$\mathbb{H}(\Z{4})\not\cong\mathbb{L}(\Z{4})$.
\end{Lemma}
\begin{proof}
It can be explicitly computed that $\mathbb{H}(\Z{4})$ has 32 involutions while $\mathbb{L}(\Z{4})$ has 64.
\end{proof}

\begin{prop}
If $s\geq 3$  then $\mathbb{L}(\Z{2^s})\not\cong\mathbb{H}(\Z{2^s})$.
\end{prop}
\begin{proof}
Assume, on the contrary, that $\mathbb{H}(\Z{2^s}) \cong  \mathbb{L}(\mathbb{Z}/{2^s}\mathbb{Z})$ whit $s>1$. This isomorphism naturally induces an isomorphism $$\mathbb{H}(\Z{2^s})/4\mathbb{H}(\Z{2^s}) \cong  \mathbb{L}(\Z{2^s})/4\mathbb{L}(\Z{2^s}).$$ Now, $4\mathbb{H}(\Z{2^s})$ and $4\mathbb{L}(\Z{2^s})$ are, respectively, the kernels of the surjective homomorphisms
$$\mathbb{H}(\Z{2^s}) \xrightarrow{\textrm{mod}\; 4}\mathbb{H}(\Z{4}),$$
$$\mathbb{L}(\Z{2^s}) \xrightarrow{\textrm{mod}\; 4}\mathbb{L}(\Z{4}).$$
Hence, it follows that $\mathbb H(\Z{4}) \cong \mathbb L(\Z{4})$ contradicting Lemma \ref{cuat}.
\end{proof}

To end this section we will see that both $\mathbb{H}(\Z{2^s})$ and $\mathbb{L}(\Z{2^s})$ are local fields, so that they cannot be isomorphic to $\mathbb{M}_2(\Z{2^s})$. Recall that a unital ring $R$ is local if and only if $1\neq 0$ and for every $r\in R$ either $r$ or $1-r$ is a unit.

\begin{prop}
If $s\geq 1$ then $\mathbb{H}(\Z{2^s})$ and $\mathbb{L}(\Z{2^s})$ are local fields.
\end{prop}
\begin{proof}
We will only focus on $\mathbb{H}(\Z{2^s})$, the other case being completely analogous. Obviously $1\neq 0$, now assume that $z\in \mathbb{H}(\Z{2^s})$ is not a unit. This means that $\textrm{n}(z)$ is not a unit in $\Z{2^s}$; i.e., that $\textrm{n}(z)$ is even. Now, $\textrm{n}(1-z)=(1-z)\overline{(1-z)}=(1-z)(1-\bar{z})=1+\textrm{n}(z)-\textrm{tr}(z)$. Since $\textrm{tr}(z)$ is even (recall Definition \ref{trn}) it follows that $\textrm{n}(1-z)$ is odd; i.e., it is a unit in $\Z{2^s}$ and, consequently $1-z$ is a unit.
\end{proof}

\subsection{The general case}
We can summarize the previous work in the following Theorem.

\begin{Theorem}
Let $n$ be an integer. Then:
\begin{itemize}
\item[i)] $\mathbb{H}(\Z{n}) \cong \mathbb{L}(\Z{n}) \cong \mathbb{M}_2(\Z{n})$, if $n$ is odd.
\item[ii)] $\mathbb{H}(\Z{n}) \cong \mathbb{L}(\Z{n}) \not\cong \mathbb{M}_2(\Z{n})$, if $n\equiv 2\pmod{4}$.
\item[iii)] $\mathbb{M}_2(\Z{n})\not\cong \mathbb{H}(\Z{n}) \ncong \mathbb{L}(\Z{n}) \ncong \mathbb{M}_2(\Z{n})$, if $n\equiv 0\pmod{4}$.
\end{itemize}
\end{Theorem}

\section{Quaternions rings over $\mathbb{Z}_n$ whit $n$ odd}
\label{SEC:ODD}

It is well-known that every quaternion algebra over a finite field $\mathbb{F}_q$ of characteristic not two splits; i.e., it is isomorphic to the matrix ring of $\mathbb{M}_2(\mathbb{F}_q)$. This is a consequence of two classical theorems by Wedderburn: the Structure theorem on finite dimensional simple algebras over a field and Wedderburn's little theorem. The following theorem generalizes this result to quaternion algebras over the ring of integers modulo an odd integer $n$. Again, the natural isomorphism (\ref{F1}) allows us to consider only the prime-power case.

\begin{Theorem}
Let $p$ be an odd prime number and let $s\geq 1$. Then, all quaternion algebras defined over the ring of residual classes $\Z{p^s}$ are isomorphic. Moreover, all quaternion algebras defined over $\Z{p^s}$ split; i.e., they are isomorphic to the matrix ring $\mathbb{M}_2(\Z{p^r})$.
\end{Theorem}
\begin{proof}
Let $a,b$ be units in $\Z{p^s}$. Due to Proposition \ref{solp} we can find $u,v \in \Z{p^s}$ such that $b \equiv u^2 - a v^2 \pmod{p^s}$.

Now, let us consider the matrices
$$I=\begin{pmatrix} 1 & 0 \\ 0 & 1 \\\end{pmatrix},\quad A=\begin{pmatrix}0 & a \\ 1 & 0 \\\end{pmatrix}, \quad B=\begin{pmatrix} u & -a v \\v& -u \\ \end{pmatrix}.$$
Clearly we have that $A^2=aI$, $B^2=bI$ and $AB=-BA$.

Moreover, if $\alpha,\beta,\gamma,\delta\in\Z{p^s}$ and we solve the linear system of equation associated to
$$x_0 I+ x_1 A + x_2 B + x_3 AB= \begin{pmatrix}\alpha & \beta \\ \gamma & \delta \\ \end{pmatrix}$$
we get that the unique solution is given by:
\begin{align*}
x_0&=\frac{\alpha+\delta}{2},\\x_1&=\frac{\beta+ a\gamma}{2 a},\\ x_2&= \frac{ \alpha\ u - \delta  u + \beta  v  - a \gamma  v }{2b},\\ x_3&= \frac{ -\beta  u + a \gamma  u - a  \alpha  v + a  \delta   v }{-2 a b}.\end{align*}
Consequently, the set $\{I,A,B,AB\}$ is a basis of $\mathbb{M}_2(\Z{p^s})$ and the result follows.
\end{proof}

\section{Quaternions rings over $\Z{n}$ with $n$ a power of two}
\label{SEC:POWT}

The first result of this section shows that, in order to study the structure of $\left(\frac{a,b}{\mathbb{Z}_{2^s}}\right)$, we can restrict ourselves to the cases when $\{a,b\} \subset \{-1,1,3,5\}$.

\begin{Lemma} \label{red}
Let $a,b,a',b'$ be odd integers such that $a \equiv a' \pmod{8}$ and  $b \equiv b' \pmod{8}$. Then,
$$\left(\frac{a,b}{\Z{2^s}}\right) \cong \left(\frac{a',b'}{\Z{2^s}}\right)$$
\end{Lemma}
\begin{proof}
Let $\alpha$ be a solution to the congruence $a'x^2\equiv a \pmod{2^s}$ and let $\beta$ be a solution to the congruence $b'y^2\equiv b\pmod{2^s}$ (they exist due to Lemma \ref{lem4}). denote by $\alpha^{-1}$ and $\beta^{-1}$ their inverses modulo $2^s$. Now, considering $i'=\alpha^{-1}i$ and $j'=\beta^{-1}j$ we have that $i'^2=a'$, $j'^2=b'$ and $i'j'=-j'i'$. Since the set $\{1,i',j',i'j'\}$ is obviously a basis of $\left(\frac{a',b'}{\mathbb{Z}_{2^s}}\right)$, the result follows.
\end{proof}

\begin{prop}
Let $a,b$ be odd integers. Then:
\begin{itemize}
\item[i)] $\left(\frac{a,b}{\Z{2^s}}\right) \cong \left(\frac{-1,b}{\Z{2^s}}\right) \cong \left(\frac{-1,a}{\Z{2^s}}\right)$, if $ab \equiv 1 \pmod{8}$.
\item[ii)] $\left(\frac{a,b}{\Z{2^s}}\right) \cong \left(\frac{1,b}{\Z{2^s}}\right) \cong \left(\frac{1,a}{\Z{2^s}}\right)$, if $ab \equiv -1 \pmod{8}$.
\end{itemize}
\end{prop}
\begin{proof}
It is enough to apply the previuos lemma together with the well-known fact that
$$\left(\frac{a,b}{\Z{2^s}}\right) \cong \left(\frac{ -ab,a}{\Z{2^s}}\right) \cong \left(\frac{ -ab,b}{\Z{2^s}}\right).$$
\end{proof}

This result leads to the following isomorphisms.

\begin{cor}
\begin{align*}\left(\frac{ 5,5}{\Z{2^s}}\right) &\cong \left(\frac{ -1,5}{\Z{2^s}}\right),\\ \left(\frac{3,3}{\Z{2^s}}\right) &\cong \left(\frac{-1,3}{\Z{2^s}}\right),\\ \left(\frac{3,5}{\Z{2^s}}\right) &\cong \left(\frac{ 1,3}{\Z{2^s}}\right) \cong \left(\frac{ 1,5}{\Z{2^s}}\right).\end{align*}
\end{cor}

The following series of propositions describe more isomorphisms.

\begin{prop}
Let $\beta\in\{-1,1,3,5\}$. Then,
$$\left(\frac{1,\beta}{\Z{2^s}}\right) \cong \mathbb{L}(\Z{2^s}).$$
\end{prop}
\begin{proof}
Given $\beta\in\{-1,1,3,5\}$, there exist integer $\eta,\theta$ such that $-\eta^2 + \theta^2 = \beta$. Let us consider matrices
$$I=\begin{pmatrix}1 & 0 \\0 & 1 \end{pmatrix},\quad A=\begin{pmatrix} 0 & i \\-i & 0 \end{pmatrix},\quad  B=\begin{pmatrix}\eta i & \theta \\\theta & -\eta i \end{pmatrix}.$$
Clearly we have that $A^2= I$, $B^2=\beta I$ and $AB=-BA$.

Moreover, the linear system of equations associated to
$$\begin{pmatrix}\alpha-\beta i & \gamma +\delta i \\\gamma -\delta i & \alpha+\beta i\end{pmatrix}= x I + y A + z B + t AB$$
has the following unique solution:
\begin{align*} x&=\alpha,\\ y&=\delta,\\ z&=  \frac{ \beta\,\eta   + \gamma\,\theta}{\eta^2 + \theta^2},\\ t&=  \frac{ \gamma\,\eta   - \beta\,\theta}{\eta^2 + \theta^2}.\end{align*}
Hence, the set $\{I,A,B,AB\}$ is a basis of $ \mathbb{L}(\Z{2^s})$ and the result follows.
\end{proof}

\begin{prop}
$$\left(\frac{-1,5}{\Z{2^s}}\right)\cong \mathbb{L}(\Z{2^s}).$$
\end{prop}
\begin{proof}
Let $\{1,i,j,k\}$ and $\{1',i',j',k'\}$ be the canonical basis of $\left(\frac{ -1,1}{\Z{2^s}}\right)$ and $\left(\frac{ -1,5}{\Z{2^s}}\right)$, respectively. Given $(\eta,\theta)$ a solution to $5x^2 + 5y^2\equiv 1\pmod{2^s}$ (it exists due to Lemma \ref{lem10}), we can define a linear transformation $\phi:\left(\frac{ -1,5}{\Z{2^s}}\right)\longrightarrow \left(\frac{ -1,1}{\Z{2^s}}\right)$ by:
$$\phi(1')=1,\quad \phi(i')=i,\quad \phi(j')=\eta j+\theta k,\quad \phi(k')= \eta k-\theta j.$$

It is easily seen that $\phi$ is in fact a well-defined algebra homomorphism, and since the set $\{1,i,\eta j+\theta k,\eta k-\theta j\}$ is linearly independent, the proof is complete.
\end{proof}

\begin{prop}
$$\left(\frac{-1,3}{\Z{2^s}}\right) \cong \mathbb{H}(\Z{2^s}).$$
\end{prop}
\begin{proof}
Note that, due to Lemma \ref{red} $\left(\frac{-1,3}{\Z{2^s}}\right)\cong \left(\frac{-1,-5}{\Z{2^s}}\right)$. Hence we can proceed in a similar way as in the previous proposition.
\end{proof}

All the previous results can be summarized in the following theorem.

\begin{Theorem}
$$\left(\frac{a,b}{\Z{2^s}}\right) \cong \begin{cases}\mathbb{H}(\Z{2^s}), & \textrm{if $a \equiv b\equiv -1 \pmod{4}$ }; \\ \mathbb{L}(\Z{2^s}), & \textrm{otherwise}.\end{cases} $$
\end{Theorem}

\section{Conclusions}

In this short final section we present the main result of the paper, which collects all our previous work. It describes the structure of quaternion rings over $\Z{n}$.

\begin{Theorem}
Sea $n$ be an integer and let $a,b$ be such that $\gcd(a,n)=\gcd(b,n)=1$. The following hold:
\begin{itemize}
\item[i)] If $n$ is odd, then
$$\left(\frac{a,b}{\Z{n}} \right) \cong \mathbb{M}_2(\Z{n}).$$
\item[ii)] If $n=2^sm$ with $s>0$ and $m$ odd, then
$$\left(\frac{a,b}{\Z{n}}\right) \cong \begin{cases}\mathbb{M}_2(\Z{m}) \times (\frac{-1,-1}{\Z{2^s}}), & \textrm{if $s=1$ or $a \equiv b\equiv -1 \pmod{4}$}; \\ \mathbb{M}_2(\Z{m}) \times (\frac{1,1}{\Z{2^s}}), & \textrm{otherwise}.\end{cases}$$
\end{itemize}
\end{Theorem}

We can restate the result in the following terms.

\begin{Theorem}
$$\left(\frac{a,b}{\Z{n}} \right) \cong \begin{cases} \mathbb{H}(\Z{n}), & \textrm{if $a \equiv b\equiv -1 \pmod{4}$ }; \\ \mathbb{L}(\Z{n}), & \textrm{otherwise}.\end{cases} $$
\end{Theorem}

In conclusion, quaternion algebras over the ring $\mathbb{Z}/n\mathbb{Z} $ split; i.e., are isomorphic to the matrix ring $\mathbb{M}_2(\Z{n})$ if and only if $n$ is odd. Moreover, for a given $n$, there are at most two isomorphism classes of quaternion algebras over $\Z{n}$ and there is only one isomorphism class if and only if $n \not \equiv 0 \pmod{4}$.

\section*{Acknowledgement}

We thank Alberto Elduque for useful comments and remarks.

\end{document}